\documentclass[10pt]{amsart}
\usepackage{hyperref}
\usepackage{pb-diagram}
\usepackage{amsmath, amssymb}
\usepackage{tikz}
\usetikzlibrary{matrix,arrows,decorations.pathmorphing}
\usepackage{paralist}

\numberwithin{equation}{section}
\newtheorem{thm}[equation]{Theorem}
\newtheorem{lem}[equation]{Lemma}
\newtheorem{prp}[equation]{Proposition}

\theoremstyle{definition}
\newtheorem{df}[equation]{Definition}
\newtheorem{exa}[equation]{Example}

\theoremstyle{remark}
\newtheorem*{rem}{Remark}

\DeclareMathOperator{\im}{im}

\DeclareMathOperator{\supp}{supp}

\DeclareMathOperator{\hocolim}{hocolim}

\DeclareMathOperator{\Tr}{Tr}

\DeclareMathOperator{\sgn}{sgn}

\DeclareMathOperator{\Ch}{Ch}
\DeclareMathOperator{\Seq}{Seq}

\def\hTop{\mathbf{hTop}}
\def\Set{\mathbf{Set}}

\def\vP{\vec{P}}

\def\R{\mathbb{R}}
\def\Z{\mathbb{Z}}

\def\Set{\mathbf{Set}}

\def\cN{\mathcal{N}}
\def\bn{\mathbf{n}}
\def\bc{\mathbf{c}}
\def\fS{\mathfrak{S}}
\def\fD{\mathfrak{D}}

\def\cN{\mathcal{N}}

\def\sfp{\mathsf{p}}
\def\sfd{\mathsf{d}}
\def\sfs{\mathsf{s}}
\def\sft{\mathsf{t}}
\def\sfu{\mathsf{u}}
\def\pcSqSet{\mathbf{pc}\square\Set}

\title{Spaces of directed paths on pre-cubical sets}
\author{Krzysztof Ziemia\'nski}
\thanks{Faculty of Mathematics, Informatics and Mechanics, University of Warsaw, Banacha 2, 02--097 Warszawa, Poland. E-mail: ziemians@mimuw.edu.pl.}

\begin{document}

\begin{abstract}
	The spaces of directed paths on the geometric realizations of pre-cubical sets, called also $\square$--sets, can be interpreted as the spaces of possible executions of Higher Dimensional Automata, which are models for concurrent computations. In this paper we construct, for a sufficiently good pre-cubical set $K$, a CW-complex $W(K)_v^w$ that is homotopy equivalent to the space of directed paths between given vertices $v$, $w$ of $K$. This construction is functorial with respect to $K$, and minimal among all functorial constructions. Furthermore, explicit formulas for incidence numbers of the cells of $W(K)_v^w$ are provided.
\end{abstract}

\maketitle

\section{Introduction}

In recent years, much effort was made to understand spaces of directed paths on d-spaces. Particularly interesting examples of d-spaces are geometric realizations of pre-cubical sets \cite{FGR}, thanks their to applications in concurrency --- their directed path spaces can be interpreted as the executions spaces of Higher Dimensional Automata \cite{Pr}. Raussen \cite{R1} proved that such spaces, under certain conditions, have the homotopy types of CW-complexes. In this paper we construct a combinatorial model of the space of directed paths between fixed two vertices of a pre-cubical set, assuming it satisfies a certain mild condition, i.e., its non-looping length covering of proper (cf.\ \ref{e:Proper}, \ref{e:NonLoopingLengthCovering}).


\emph{A pre-cubical set} $K$, called also a $\square$--set, is a sequence of disjoints sets $K[n]$, for $n\geq 0$, equipped with \emph{face maps} $d^\varepsilon_i:K[n]\to K[n-1]$, for $\varepsilon\in\{0,1\}$ and $i\in\{1,\dots,n\}$, that satisfy pre-cubical relations; namely, $d^\varepsilon_i d^\eta_j=d^{\eta}_{j-1}d^\varepsilon_i$ for $i<j$.
\emph{A $\square$-map $f:K\to K'$} between $\square$-sets is a sequence of maps $f[n]:K[n]\to K[n']$ that commute with the face maps. 
Elements of $K[n]$ will be called \emph{$n$-cubes} of $K$; in particular $0$-cubes will be called \emph{vertices}. A  \emph{bi-pointed} $\square$--set is a triple $(K,v,w)$, where $K$ is a $\square$--set, and $v,w\in K[0]$ are its vertices. Let $\square\Set$ and $\square\Set_*^*$ denote the category of $\square$--sets and $\square$--maps and the category of bi-pointed $\square$--sets and $\square$--maps preserving the distinguished vertices, respectively.

Let us introduce a notation for arbitrary compositions of face maps. For a function $f:\{1,\dots,n\}\to \{0,1,*\}$ such that $|f^{-1}(*)|=m$, define a map $d_f:K[n]\to K[m]$ by
\begin{equation}
	d_f=d_1^{f(1)}d_2^{f(2)}\dots d_{n}^{f(n)},
\end{equation}
where $d_i^*$ is, by convention, the identity map; we will also write $d_{f^{-1}(0),f^{-1}(1)}$ for $d_f$. Finally, let $d^0_A:=d_{A,\emptyset}$, $d^1_A=d_{\emptyset,A}$ and $d^\varepsilon=d^\varepsilon_{\{1,\dots,n\}}:K[n]\to K[0]$. If $f:\{1,\dots,n\}\to\{0,1,*\}$, $g:\{1,\dots,m\}\to\{0,1,*\}$ are functions such that $|f^{-1}(*)|=m$, $|g^{-1}(*)|=k$, then $d_gd_f=d_h$, where
\begin{equation}\label{e:DfComposition}
	h:\{1,\dots,n\}\ni i \mapsto\begin{cases}
		f(i) & \text{for $i\not\in f^{-1}(*)$}\\
		g(\alpha(i))& \text{for $i\in f^{-1}(*)$}
	\end{cases}
	\in\{0,1,*\},
\end{equation}
and $\alpha:f^{-1}(*)\to\{1,\dots,m\}$ as the unique increasing bijection.


The results of this paper apply to finite $\square$--sets that have proper non-looping length coverings. We say that a $\square$-set $K$ is \emph{proper}, if the map
\begin{equation}\label{e:Proper}
	\coprod_{n\geq 0} K[n]:c\mapsto \{d^0(c),d^1(c)\}\in 2^{K[0]}
\end{equation}
is an injection, i.e., the cubes of $K$ can be distinguished by its extreme vertices. Every proper $\square$-set is non-self-linked in the sense of \cite{FGR}, i.e., for every $c\in K[n]$ and $f,g:\{1,\dots,n\}\to\{0,1,*\}$, the equity $d_f(c)=d_g(c)$ implies that $f=g$. \emph{The non-looping length covering} \cite[5.1]{R3} of a $\square$--set $K$ is a $\square$--set $\tilde{K}$ such that $\tilde{K}[n] = K[n]\times \Z$ and
\begin{equation}\label{e:NonLoopingLengthCovering}
	d^\varepsilon_i(c,k)=(d^\varepsilon_i(c),k+\varepsilon).
\end{equation}
Clearly, if a $\square$--set is proper, then its non-looping length covering is also proper. Let $\pcSqSet_*^*\subseteq \square\Set_*^*$ be the full subcategory of finite bi-pointed $\square$--sets having the proper non-looping covering.

\begin{df}
	Let $K$ be a $\square$-set and let $v,w\in K[0]$ be its vertices. \emph{A cube chain in $K$ from $v$ to $w$} is a sequence of cubes $\mathbf{c}=(c_1,\dots,c_l)$, where $c_k\in K[n_k]$ and $n_k>0$, such that
	\begin{itemize}
		\item{$d^0(c_1)=v$,}
		\item{$d^1(c_l)=w$,}
		\item{$d^1(c_i)=d^0(c_{i+1})$ for $i=1,\dots,l-1$.}
	\end{itemize}
	The sequence $(n_1,\dots,n_l)$ will be called \emph{the type} of a cube chain $\mathbf{c}$, $\dim(\mathbf{c})=n_1+\dots+n_l-l$ \emph{the dimension} of $\mathbf{c}$, and $n_1+\dots+n_l$ \emph{the length} of $\mathbf{c}$. The set of all cube chains in $K$ from $v$ to $w$ will be denoted by $\Ch(K)_v^w$, and the set of cube chains of dimension equal to $m$ (resp.\ less than $m$, less or equal to $m$) by $\Ch^{=m}(K)_v^w$ (resp.\ $\Ch^{<m}(K)_v^w$, $\Ch^{\leq m}(K)_v^w$).
\end{df}

For a cube chain $\bc=(c_1,\dots,c_l)\in \Ch(K)_v^w$ of type $(n_1,\dots,n_l)$, an integer $k\in \{1,\dots,l\}$ and a subset $A\subsetneq \{1,\dots, n_k\}$ having $r$ elements, where $0<r<n_k$, define a cube chain 
\begin{equation}\label{e:FaceCubeChain}
	d_{k,A}(\bc) = (c_1,\dots,c_{k-1},d^0_{\bar{A}}(c_k), d^1_{A}(c_k),c_{k+1},\dots,c_l)\in \Ch(K)_v^w,
\end{equation}
where $\bar{A}=\{1,\dots,n_k\}\setminus A$. Let $\leq$ be the partial order on $\Ch(K)$ spanned by all relations having the form $d_{k,A}(\bc)\leq \bc$ (the relation $\leq$ is antisymmetric since $d_{k,A}(\bc)$ has more cubes than $\bc$).

For a $\square$-set $K$ let $|K|$ denote its geometric realization (\ref{e:GeometricRealization}) and, for a bi-pointed $\square$--set $(K,v,w)$, let $\vec{P}(K)_v^w$ be the space of directed paths on $|K|$ from $v$ to $w$ (cf.\ \ref{d:dSpace}). Clearly, $(K,v,w)\mapsto \Ch(K)_v^w$ and $(K,v,w)\mapsto \vP(K)_v^w$ are functors from $\square\Set_*^*$ into the categories of posets and topological spaces, respectively.

Let $\hTop$ be the homotopy category of the category of topological spaces. We prove the following

\begin{thm}\label{t:Cover}
	The functors
\[
	\vP:\pcSqSet_*^*\ni (K,v,w)\mapsto \vP(K)_v^w\in \hTop
\]	
and
\[
	|\Ch|:\pcSqSet_*^*\ni (K,v,w)\mapsto |\Ch(K)_v^w|\in \hTop
\]
are naturally equivalent. In other words, for every finite bi-pointed $\square$--set $(K,v,w)$ having the proper non-looping covering, there exists a homotopy equivalence $\varepsilon_{(K,v,w)}:\vec{P}(K)_v^w\to |\Ch(K)_v^w|$ such that, for every $\square$--map $f:K\to K'$, the diagram
	\[
		\begin{diagram}
			\node{\vec{P}(K)_v^w}
				\arrow[3]{e,t}{\alpha\mapsto |f|\circ\alpha}
				\arrow{s,l}{\varepsilon_{(K,v,w})}
			\node{}
			\node{}
			\node{\vec{P}(K')_{f(v)}^{f(w)}}
				\arrow{s,r}{\varepsilon_{(K',f(v),f(w))}}
		\\
			\node{|\Ch(K)_v^w|}
				\arrow[3]{e,t}{Ch(f):\bc\mapsto f(\bc)}
			\node{}
			\node{}
			\node{|\Ch(K')_{f(v)}^{f(w)}|}
		\end{diagram}
	\]
	commutes up to homotopy.	
\end{thm}

Next, the prove that the spaces $\Ch(K)_v^w$ have a natural CW-structure, which is coarser than the simplicial one. For a poset $P$ and $x\in P$, let $P_{\leq x}$ (resp.\ $P_{<x}$) be the subposet of $P$ containing all elements that are less or equal (resp.\ equal) to $x$.

\begin{thm}\label{t:CW}
	Let $(K,v,w)$ be a finite bi-pointed $\square$--set having the proper non-looping covering (i.e., $(K,v,w)\in\pcSqSet$). The space $|\Ch(K)_v^w|$ is a regular CW-complex with $d$--dimensional cells having the form $|\Ch_{\leq \bc}(K)|$ for $\bc\in \Ch^{=d}(K)_v^w$. Furthermore, for every $\square$--map $f:K\to K'$, the induced map $f_*:|\Ch(K)_v^w|\to|\Ch(K')_{f(v)}^{f(w)}|$ is cellular.
\end{thm}

We also calculate the incidence numbers between the cells of this CW-complex. 
In Section \ref{s:CWDecomposition} we construct, for every $\bc\in \Ch^{=n}(K)_v^w$, a cycle $g_\bc$ that represents a generator in the simplicial homology group
\begin{equation}
	[g_\bc]\in H_n(|\Ch_{\leq \bc}(K)|_v^w, |\Ch_{<\bc}(K)|_v^w),
\end{equation}
As a consequence of $\ref{t:CW}$, the group $H_n(|\Ch^{\leq n}(K)_v^w|, |\Ch^{<n}(K)_v^w|)$ is a free group generated by $g_\bc$ for $\bc\in \Ch^{=n}(K)_v^w$.

\begin{thm}\label{t:Diff}
	Let $(K,v,w)$ be a finite bi-pointed $\square$--set having the proper non-looping covering, and let 
\[
	\partial_n:H_n(|\Ch^{\leq n}(K)_v^w|, |\Ch^{<n}(K)_v^w|)\to H_{n-1}(|\Ch^{\leq n-1}(K)_v^w|, |\Ch^{<n-1}(K)_v^w|)
\]	
be the differential in the cellular chain complex of $|\Ch(K)_v^w|$ (cf.\ \cite[p.\ 139]{H}). For a cube chain $\bc\in\Ch^{=n}(K)_v^w$ of type $(n_1,\dots,n_l)$ we have
	\[
		\partial(g_\bc)=\sum_{k=1}^{l} \sum_{r=1}^{n_k-1} \sum_{A\subseteq\{1,\dots,n_k\}:\; |A|=r} (-1)^{n_1+\dots+n_{k-1}+k+r+1}\sgn(A) g_{d_{k,A}(\bc)},
	\]
	where 
	\[
		\sgn(A)=\begin{cases}
			1 & \text{if $\sum_{i\in A}i\equiv \sum_{i=1}^r i \mod 2$,}\\
			-1 & \text{otherwise.}
		\end{cases}
	\]
\end{thm}

As a consequence, the homology of $\vP(K)_v^w$ can be calculated using the formula above.

\section{Directed spaces}

In this Section, we recall a notion of d-space and introduce complete d-spaces, which are necessary to define the d-structure on the geometric realization of a d-simplicial complex. The following definitions are due to Grandis \cite{Gr}.

\begin{df}\label{d:dSpace}
	Let $X$ be a topological space and let $P(X)$ be its path space, i.e.\ the space of continuous maps $[0,1]\to X$ with compact-open topology.
	\begin{itemize}
	\item{\emph{A d-structure} on $X$ is a subset $\mathfrak{D}\subseteq P(X)$ which contains all constant paths, and is closed with respect to concatenations and non-decreasing reparametrizations.}
	\item{\emph{A d-space} is a pair $(X,\mathfrak{D})$, where $X$ is a topological space and $\mathfrak{D}$ is a d-structure on $X$.}
	\item{Let $(X,\fD)$, $(X',\fD')$ be d-spaces. A continuous map $f:X\to X'$ is \emph{a d-map} if it preserves d-structure, i.e.\ $f\circ\alpha\in \fD'$ for each $\alpha\in\fD$.}
	\item{\emph{A d-homeomorphism} is an invertible d-map.}
	\end{itemize}
\end{df}

The category of d-spaces and d-maps is complete and cocomplete. For an arbitrary family of paths $\mathfrak{S}\subseteq P(X)$ there exists the smallest d-structure $\overline{\fS}$ on $X$ containing $\fS$; it can be constructed as the intersection of all d-structures containing $\fS$, or by adding to $\fS$ all constant paths and all non-decreasing reparametrizations of concatenations of paths in $\fS$.

The following d-spaces play an important role in this paper:
\begin{itemize}
\item{\emph{the directed real line} $\vec{\R}=(\R,  \text{non-decreasing paths})$,}
\item{\emph{the directed $n$-cube} $\vec{I}^n=([0,1]^n, \text{paths with non-decreasing coordinates})$,}
\item{\emph{the directed $n$-simplex} $\vec{\Delta}^n=(\Delta^n,\vec{P}(\Delta^n))$, where
\[
	{\Delta}^n=\left\lbrace(t_0,\dots,t_n)\in [0,1]^{n+1}:\; \sum_{i=0}^n t_i=1\right\rbrace
\]
and $\alpha=(\alpha_0,\dots,\alpha_n)\in P(\Delta^n)$ is directed if and only if the functions 
\[
	[0,1]\ni t \mapsto \sum_{i\geq k}\alpha_k(t)\in [0,1]
\]
are non-decreasing for $k\in\{1,\dots,n\}$.
}
\end{itemize}

Let us recall the definition of complete d-spaces introduced in \cite{Z2} (called there "good" d-spaces).

\begin{df}
	Let $(X,\fD)$ be a d-space. A path $\alpha:[0,1]\to X$ is \emph{almost directed} if for every open subset $U\subseteq X$ and every $0\leq s<t \leq 1$ such that $\alpha([s,t])\subseteq U$, there exists a directed path $\beta\in\vec{P}(X)$ such that $\beta(0)=\alpha(s)$, $\beta(1)=\alpha(t)$ and $\beta([0,1])\subseteq U$.
\end{df}

The family of all almost directed paths $\mathfrak{D}_c$ on a d-space $(X,\mathfrak{D})$ is a d-structure on $X$. The d-space $(X,\fD_c)$ will be called \emph{the completion} of $(X,\fD)$; if a given d-space is equal to its completion, it is called \emph{complete}. If $\mathfrak{S}\subseteq P(X)$ is  an arbitrary family of paths, then its \emph{completion} $\overline{\fS}_c$ is the smallest complete d-structure containing $\fS$. Here follows the criterion which allows to verify whether a continuous map between complete d-spaces is a d-map.

\begin{prp}\label{p:CompletenessCriterion}
	Let $f:X\to X'$ be a continuous map and let $\fS\subseteq P(X)$, $\fS'\subseteq P(X')$ be families of paths. Assume that $f\circ \alpha\in \overline{\fS'}_c$ for every $\alpha\in \mathfrak{S}$. Then the map $f:(X,\overline{\fS}_c)\to (Y,\overline{\fS'}_c)$ is a d-map.
\end{prp}
\begin{proof}
	Fix a path $\alpha\in \overline{\fS}_c$, $0\leq s<t\leq 1$ and an open subset $U\subseteq Y$ such that $f(\alpha([s,t])\subseteq U$. Since $\alpha$ is almost directed with respect to $\overline{\fS}$, there exists a path $\beta\in \overline{\fS}$ such that $\beta(0)=\alpha(s)$, $\beta(1)=\alpha(t)$ and $\beta([0,1])\subseteq f^{-1}(U)$. The path $\beta$ is either constant, or it is a non-decreasing reparametrization of concatenation of paths in $\fS$. Thus, by assumption, $f\circ \beta$ is either constant, or it is an non-decreasing reparametrization of concatenation of paths in $\fS'$. This implies that $f\circ\alpha$ is almost directed with respect to $\overline{\fS'}$; therefore $f\circ\alpha\in\overline{\fS}_c$.
\end{proof}

For every $n\geq 0$, the directed $n$-cube and the directed $n$-simplex are complete d-spaces. The d-structure on $\vec{I}^n$ is generated by a family of paths
\begin{equation}\label{e:CubeGeneratingPaths}
	\{[0,1]\ni t\mapsto (x_1,\dots,x_{k-1},t,x_{k+1},\dots,x_n):\; k\in\{1,\dots,n\},\; x_i\in [0,1]\}
\end{equation}
and the d-structure on $\vec{\Delta}^n$ by
\begin{multline}\label{e:SimplexGeneratingPaths}
	\{[0,1]\ni t\mapsto (x_0,\dots,x_{k-2},(1-t)c,tc,x_{k+1},\dots,x_n):\\
	\; k\in \{1,\dots,n\},\; c,x_i\in [0,1],\; c+\sum_{i\neq k-1,k}{x_i}=1 \}.
\end{multline}

If it is clear which d-structure we consider on a given space $X$, we will denote it by $\vec{P}(X)$; the elements of $\vec{P}(X)$ will be called \emph{directed paths} or \emph{d-paths}. Given two points $x,y\in X$, let
\begin{equation}
	\vec{P}(X)_x^y:=\{\alpha\in \vec{P}(X):\; \alpha(0)=x,\; \alpha(1)=y\}
\end{equation}
be the space of directed  paths from $x$ to $y$. We extend a notion of directedness for paths defined on arbitrary closed intervals; a path $[a,b]\to X$ is called directed if its linear reparametrization $[0,1]\ni t\mapsto \alpha(a+t(b-a))\in X$ is directed. The space of such paths will be denoted by $\vec{P}_{[a,b]}(X)$.

\section{$\square$-sets}

In this Section we introduce $\square$-sets with height function, and prove that their geometric realizations are complete d-spaces. For a more detailed discussion on $\square$--sets see also \cite{FGR}.

\emph{The geometric realization} of a $\square$-set $K$ is a d-space
\begin{equation}\label{e:GeometricRealization}
	|K|=\coprod_{n\geq 0} K[n]\times \vec{I}^n/(d^\varepsilon_i(c),x)\sim (c,\delta^\varepsilon_i(x)),
\end{equation}
where $\delta^\varepsilon_i(s_1,\dots,s_{n-1})=(s_1,\dots,s_{i-1},\varepsilon, s_i,\dots,s_{n-1})$. A path $\alpha\in P(|K|)$ is directed if there exist numbers $0=t_0<t_1<\dots<t_l=1$, cubes $c_i\in K[n_i]$ and directed paths $\beta_i:[t_{i-1},t_i]\to \vec{I}^{n_i}$ such that $\alpha(t)=(c_i,\beta_i(t))$ for $t\in [t_{i-1},t_i]$. For every $x\in |K|$ there exists a unique cube $\supp(x)$ of $K$, called \emph{the support of $x$}, such that $x=(\supp(x),(t_1,\dots,t_n))$, and $t_i\neq 0,1$ for $i\in\{1,\dots,n\}$. For shortness, we will further write $\vP(K)$ instead of $\vP(|K|)$.

\begin{df}
	\emph{A height function} on a $\square$-set $K$ is a function $h:\coprod_{n\geq 0} K[n]\to \Z$ such that $h(d^\varepsilon_i(c))=h(c)+\varepsilon$ for every $n\geq 0$, $c\in K[n]$, $\varepsilon\in \{0,1\}$ and $i\in\{1,\dots,n\}$.
\end{df}

A height function $h$ on $K$ determines a d-map 
	\begin{equation}
		h:|K|\ni (c,(s_1,\dots,s_n))\mapsto h(c)+\sum_{i=1}^n s_i\in \vec{\R}
	\end{equation}
which will be also called a height function. If $\alpha\in\vec{P}(K)$ is a directed path, then $l_1(\alpha)=h(\alpha(1))-h(\alpha(0))$, where $l_1$ is the $L_1$-arc length of Raussen \cite[5.1]{R3}. In particular, if $\alpha\in \vec{P}(K)$ is not constant on any interval, then $h\circ\alpha$ is a strictly increasing function, which implies that all directed loops on a $\square$-set with height function are constant.

\begin{df}
	Let $K$ be a $\square$-set with height function. A d-path $\alpha:[a,b]\to |K|$ is \emph{natural} if $h(\alpha(t))=t$ for every $t\in [a,b]$.
	For $v,w\in K[0]$, let $\vec{N}(K)_v^w\subseteq \vec{P}_{[h(v),h(w)]}(|K|)_v^w$ be the space of natural paths from $v$ to $w$.
\end{df}

The space $\vec{N}(K)_v^w$ can be regarded as a subspace $\vec{P}(K)_v^w$ via the reparametrizing inclusion
\begin{equation}\label{e:ReparamInclusion}
	\vec{N}(K)_v^w\ni\alpha \mapsto \alpha\circ (t\mapsto a+(b-a)t) \in \vec{P}(K)_v^w.
\end{equation}
On the other hand, by \cite[2.15]{R1} there exists \emph{a naturalization map}
\begin{equation}\label{e:Nat}
	nat:\vec{P}(K)_v^w\to \vec{N}(K)_v^w,
\end{equation}
which is a left inverse of (\ref{e:ReparamInclusion}). As shown in \cite{R1}, these maps are mutually inverse homotopy equivalences.

 We conclude this Section with the following

\begin{prp}\label{p:CubeSetIsComplete}
	If $K$ is a $\square$-set with height function, then its geometric realization $|K|$ is a complete d-space.
\end{prp}
\begin{proof}
	Let $\preceq$ be the minimal reflexive and transitive relation on $\coprod_{n\geq 0} K[n]$ such that  $d^0_i(c)\preceq c$ and $c\preceq d^1_i(c)$ for $n\geq 0$, $c\in K[n]$, $i\in \{1,\dots,n\}$. The relation $\preceq$ is a partial order since, for $c\neq c'$, $c\preceq c'$ implies that either $h(c)<h(c')$, or $h(c)=h(c')$ and $\dim(c)<\dim(c')$. Moreover, immediately from definition follows that $\supp(\alpha(0))\preceq \supp(\alpha(1))$ for every directed path $\alpha$ on $|K|$.
	
	 Let $\alpha$ be an almost directed path in $|K|$ and let, for an arbitrary cube $c$ of $K$,   $J_c\subseteq [0,1]$ be the closure of the set $\{s\in [0,1]:\; \supp(\alpha(s))=c\}$. For $s<s'$ there exists a directed path from $\alpha(s)$ to $\alpha(s')$; therefore $\supp(\alpha(s))\preceq \supp(\alpha(s'))$. This implies that all the sets $J_c$ are either closed intervals or empty. Thus $\alpha$ is a concatenation of almost directed paths all of which lie in a single cube. Every such path is directed, then so is $\alpha$. 	
\end{proof}

\section{d-simplicial complexes}

In this Section, we recall the definition of d-simplicial complex and prove that every proper $\square$-set with height function has a d-simplicial triangulation.

\begin{df}[{\cite{Z1}}]
	\emph{A d-simplicial complex} $M$ is a triple $(V_M,S_M,\leq_M)$, where $(V_M,S_M)$ is a simplicial complex and $\leq_M$ is a binary relation of $V_M$ such that
	\begin{itemize}
	\item{for every simplex  $A\in S_M$, the restriction $(\leq_M)|_A$ is a total order on $A$,}
	\item{if $v\leq_M v'$, then $\{v,v'\}\in S_M$.}
	\end{itemize}
	\emph{A d-simplicial map} $f:M\to M'$ is a simplicial map such that $v\leq_M w$ implies $f(v)\leq_{M'}f(w)$ for $v,w\in V_{M}$.
	\emph{A geometric realization} of a d-simplicial complex $M$ is the space
	\[
		|M|=\left\lbrace \sum_{v\in V_M} t_vv :\; t_v\geq 0,\; \sum_{v\in V_M} t_v=1,\; \{v:\; t_v>0\}\in S_M\right\rbrace
	\]
	with the maximal topology and the minimal complete d-structure such that the inclusions of simplices 
	\begin{equation}\label{e:InclusionOfSimplex}
		\vec{\Delta}^n\ni (t_0,\dots,t_n) \mapsto \sum_{i=0}^n t_iv_i\in |M|	
	\end{equation}
	 are d-maps for all $\{v_0<\dots<v_n\}\in S_M$.
\end{df}

\begin{rem}
	The completion of d-structure is necessary to obtain proper d-structures on the triangulations of $\square$-sets. For example, the directed square $\vec{\square}^2$ admits a triangulation that is a d-simplicial complex with vertices $(i,j)$, $i,j\in\{0,1\}$, maximal simplices $((0,0)<(0,1)<(1,1))$ and $((0,0)<(1,0)<(1,1))$, and $(i,j)\leq (i',j')$ if and only if $i\leq i'$, $j\leq j'$. However, the staircase path on the picture below
\begin{equation}
\begin{tikzpicture}
	\draw[thick] (0,0)--(4,0)--(4,4)--(0,4)--(0,0);
	\draw (0,0)--(4,4);
	\draw[very thick,->] (0,0)--(1,0)--(1,0.5);
	\draw[very thick,->] (1,0.5)--(1,2)--(2.25,2);
	\draw[very thick,->] (2.25,2)--(2.5,2)--(2.5,2.75);
	\draw[very thick] (2.5,2.75)--(2.5,3)--(3.25,3)--(3.25,3.5)--(3.625,3.5)--(3.625,3.75)--(3.8125,3.75)--(3.8125,3.875)--(3.90625,3.875)--(3.90625,3.9375);
	\draw (-0.4,0.15)node {(0,0)};
	\draw (-0.4,3.85)node {(0,1)};
	\draw (4.4,3.85)node {(1,1)};
	\draw (4.4,0.15)node {(1,0)};
\end{tikzpicture}
\end{equation}	
	is not directed with respect to the (non-completed) d-structure induced by the inclusions of simplices, though it is directed with respect to the d-structure of the square.
\end{rem}

\begin{exa}
	If $P$ is a poset, then the nerve of $P$, denoted by $\cN P$, is defined by $V_{\cN P}=P$, $\leq_{\cN P}=\leq_P$ and $A\subseteq P$ is a simplex of $\cN P$ if and only if $A$ is a totally ordered subset.
\end{exa}

\begin{rem} 
	d-simplicial complexes can be regarded as special cases of simplicial sets. For a d-simplicial complex $M$ one can define a simplicial set whose $n$-simplices are functions $\sigma:\{0,\dots,n\}\to V_M$ such that $\sigma(i)\leq_M\sigma(j)$ for every $i\leq j$ and the image of $\sigma$ is a simplex of $M$.
\end{rem}

For the rest of the Section we assume that $K$ is a proper $\square$-set. We will construct a triangulation of $K$, i.e., a d-simplicial complex $\Tr_K$ such that $|K|$ and $\Tr_K$ are d-homeomorphic. Let
\begin{multline}
	S_{\Tr(K)}:=\{\{d_{f_0}(c), \dots,d_{f_k}(c)\}\subseteq K[0]:\\
	\; c\in K[n],\; f_0<f_1<\dots <f_k,\; f_i:\{1,\dots,n\}\to \{0,1\} \},
\end{multline}
where $f_j<f_k$ means that $f_j\neq f_k$ and $f_j(i)\leq f_k(i)$ for all $i$. Introduce a binary relation $\leq_{\Tr(K)}$ on $K[0]$ by
\begin{equation}
	v\leq_{\Tr(K)} v' \;\Leftrightarrow\; \exists_{c\in K[n]} \exists_{f\leq f':\{1,\dots,n\}\to \{0,1\}}\; d_f(c)=v,\; d_{f'}(c)=v'.
\end{equation}
Every element $A=\{d_{f_0}(c)<\dots<\dots,d_{f_k}(c)\}$, $c\in K[n]$ of $S_{\Tr(K)}$ can be written as
\[
	A=\{d_{f_0\circ e}(c')<\dots<\dots,d_{f_k\circ e}(c')\},
\]
where $c'=d_{f_k^{-1}(0),f_0^{-1}(1)}(c)$ and
\[
	e:\{1,\dots,n-|f_k^{-1}(0)|-|f_0^{-1}(1)|\}\to \{1,\dots,n\}\setminus (f_k^{-1}(0)\cup f_0^{-1}(1))
\]
is an increasing bijection. Thus, $A$ has a unique presentation such that $f_0\equiv 0$ and $f_1\equiv 1$. By the properness of $K$, the cube $c$ is determined by its extreme vertices $d^0(c)=d_{f_0}(c)$ and $d^1(c)=d_{f_k}(c)$ and, for every vertex $v$ of $c$, there exists a unique $f$ such that $v=d_f(c)$. Such a presentation of a simplex of $\Tr(K)$ will be called \emph{canonical}.

\begin{prp}
	If $K$ is a proper $\square$-set, then 
	\[
		\Tr(K)=(V_{\Tr(K)}:=K[0], S_{\Tr(K)}, \leq_{\Tr(K)})
	\]
	is a d-simplicial complex. 
\end{prp}
\begin{proof}
	It follows immediately from definition that $\Tr(K)$ is a simplicial complex, and that $v\leq_{\Tr(K)} v'$ implies $\{v,v'\}\in S_{\Tr(K)}$. It remains to prove that, for every $A\in S_{\Tr(K)}$, the restriction $\leq_{\Tr(K)}|_A$ is a total order. We have
	\[
		A=\{v_0=d_{f_0}c\leq_{\Tr(K)}v_1=d_{f_1}c\leq_{\Tr(K)} \dots\leq_{\Tr(K)} v_k=d_{f_k}c\}
	\]
	for some $n\geq 0$, $c\in K[n]$ and $f_0<\dots< f_k$. Since $K$ is proper, the vertices $v_0, \dots,v_k$ are all different. Clearly $\leq_{\Tr(K)}|_A$ contains a total order. Assume that $v_j\leq_{\Tr(K)}v_i$ for $i<j$. Then there exists $c'\in K[n']$ and $f'_0,f'_1:\{1,\dots,n'\}\to \{0,1\}$, $f'_0<f'_1$, such that $d_{f'_0}c'=v_j$ and $d_{f'_1}c'=v_i$. We have
	\begin{align*}
		d^0(d_{f_j^{-1}(0), f_i^{-1}(1)}c)=v_i & \neq v_j=d^1(d_{f_j^{-1}(0), f_i^{-1}(1)}c),\\
		d^0(d_{(f'_1)^{-1}(0),(f'_0)^{-1}(1)}c')=v_j&\neq v_i =  d^1(d_{(f'_1)^{-1}(0),(f'_0)^{-1}(1)}c'),\\
	\end{align*}
	then $d_{f_j^{-1}(0), f_i^{-1}(1)}c$ and $d_{(f'_1)^{-1}(0),(f'_0)^{-1}(1)}c'$ are two different cubes with the same sets of extreme vertices, which contradicts the properness of $K$.
\end{proof}

The d-simplicial complex $\Tr(K)$ will be called \emph{the triangulation} of $K$.

Our next goal is to construct a d-homeomorphism between the geometric realization of $K$ and the geometric realization of its triangulation. Let $A$ be a simplex of $\Tr(K)$ and let $(d_{f_0}(c),\dots,d_{f_{k}}(c))$, $c\in K[n]$ be its unique presentation such that $f_0\equiv 0$, $f_k\equiv 1$. Define a continuous map
\begin{equation}
	F_A:|A|\ni \sum_{i=0}^k t_i d_{f_i}(c)\mapsto \left(c, \sum_{i=0}^k t_i f_i\right) \in |K|,
\end{equation}
where functions $f:\{1,\dots,n\}\to [0,1]$ are regarded as points $(f(1),\dots,f(n))\in \vec{I}^n$.

\begin{prp}\label{p:FCompatibility}
	If $x\in |A|\cap |A'|$, then $F_A(x)=F_{A'}(x)$.
\end{prp}
\begin{proof}
	It is sufficient to check this for
\begin{align*}
	A'&=(d_{f_0}(c),\dots,d_{f_k}(c)),\\
	A&=(d_{f_0}(c),\dots,d_{f_{j-1}}(c),d_{f_{j+1}}(c),\dots,d_{f_k}(c)),
\end{align*}	
$f_0\equiv 0$, $f_1\equiv 1$. If $j\neq 0,k$, then the preceding presentation is canonical and the equation $F_A=F_{A'}|_{|A|}$ is clearly satisfied. For $j=0$ the canonical presentation of $A$ is
	\[
		A=(d_{f_1\circ e}(d^1_{f^{-1}_0(1)}(c)),\dots,d_{f_k\circ e}(d^1_{f^{-1}_0(1)}(c))),
	\]
	where $e:\{1,\dots,n-|f^{-1}_1(1)|\}\to \{1,\dots,n\}\setminus f^{-1}_1(1)$ is an increasing bijection. We have
	\begin{multline*}
		f_{A}\left(\sum_{i=1}^k t_id_{f_i}(c)\right)
		=f_A\left(\sum_{i=1}^k t_i d_{f_i\circ e}(d^1_{f^{-1}_0(1)}(c))\right)
		=\left(d^1_{f^{-1}_0(1)}(c),\sum_{i=1}^k t_i (f_i\circ e) \right)\\
		=\left(c,\delta^1_{f^{-1}_0(1)}\left(\sum_{i=1}^k t_i (f_i\circ e)\right) \right)
		=\left(c,\sum_{i=1}^k t_i f_i\right)
		=f_{A'}\left(\sum_{i=1}^k t_id_{f_i}(c)\right).
	\end{multline*}
	The case $j=k$ is similar.
\end{proof}

As a consequence, the maps $F_A$ glue to a continuous map $F_K:|\Tr(K)|\to|K|$.

\begin{prp}
	$F_K$ is a d-homeomorphism.
\end{prp}
\begin{proof}
	For $c\in K[n]$ and a permutation $\sigma$ of $\{1,\dots,n\}$ define a set
	\[
		S_{c,\sigma}:=\{(c,(t_1,\dots,t_n))\in |K|:\; t_{\sigma(1)}\leq t_{\sigma(2)}\leq\dots\leq t_{\sigma(n)}\}
	\]
	and a sequence of functions $f^\sigma_0,\dots,f^\sigma_n:\{1,\dots,n\}\to \{0,1\}$ by
	\[
		f^\sigma_i(j)=\begin{cases} 0 & \text{for $\sigma(j)>i$} \\ 1 & \text{for $\sigma(j)\leq i$.} \end{cases}
	\]
	Clearly the sets $S_{c,\sigma}$ cover $|K|$. Similarly to the proof of \ref{p:FCompatibility} we can show that the maps
	\[
		G_{c,\sigma}: S_{c,\sigma} \ni (c,(t_1,\dots,t_n)) \mapsto \sum_{i=0}^n (t_{\sigma(i+1)}-t_{\sigma(i)}) d_{f_i^\sigma}(c)\in |\Tr(K)|,
	\]
	where by convention $t_{\sigma(0)}=0$ and $t_{\sigma(n+1)}=1$, glue to the map $G_K:|K|\to|\Tr(K)|$, which is the inverse of $F$. Thus $F_K$ is a homeomorphism. Next, we need to prove that both $F_K$ and $G_K$ are d-maps. By definition, $\vec{P}(|\Tr(K)|)$ is a complete d-structure generated by paths having the form (cf.\ \ref{e:SimplexGeneratingPaths}, \ref{e:InclusionOfSimplex})
	\[
		\omega: s\mapsto b(1-s)d_{f_k}(c)+bsd_{f_{k+1}}(c)+\sum_{i\neq k,k+1} t_i d_{f_i}(c)
	\]
	for $n\geq 0$, $c\in K[n]$, $0\equiv f_0<\dots<f_n\equiv 1$, $k\in\{0,\dots,n-1\}$, $b+\sum_{i\neq k,k+1}t_i=1$. The image
	\[
		F_K(\omega(s))=F_c(\omega(s))= b(1-s)f_k+bsf_{k+1}+\sum_{i\neq k,k+1} t_i f_i
	\]
	is a directed path in $|c|\subseteq |K|$ since $f_k<f_{k+1}$. Since $|K|$ is a complete d-space (by \ref{p:CubeSetIsComplete}) from \ref{p:CompletenessCriterion} follows that $F_K$ is a d-map. The similar argument applies for $G_K$.
\end{proof}

We conclude with some obvious properties of the triangulation.

\begin{prp}\label{p:TraingulationProperties}
	Let $K$ be a finite proper $\square$-set with height function.
	\begin{enumerate}[(1)]
	\item{$F_K$ maps bijectively vertices of $|\Tr(K)|$ into vertices of $|K|$.}
	\item{For each simplex $A\in S_{\Tr(K)}$ there exists a cube of $K$ such that $F_K(|A|)\subseteq |c|$.}
	\item{$\Tr$ is a functor from the category of proper $\square$-sets into the category of d-simplicial complexes. Moreover, the maps $F_K$ define a natural equivalence of functors $K\mapsto |K|$ and  $K\mapsto |\Tr(K)|$.\qed}
	\end{enumerate}
\end{prp}

\section{Tame paths}

Recall \cite{Z1} that a d-simplicial complex $M$ has no loops if for every sequence of its vertices $v_0\leq \dots \leq v_n\leq v_0$ we have $v_0=\dots=v_n$. We say that a $\square$-set $K$ has no loops if, for every vertex $v\in K[0]$, every cube chain from $v$ to $v$ is empty. If $K$ admits a height function, then, clearly, it has no loops. Immediately from the construction follows that the triangulation of a proper $\square$--set having no loops also has no loops. In this Section we prove that the space of directed paths between two arbitrary vertices of a proper $\square$-set having no loops is homotopy equivalent to its subspace containing only tame paths, i.e.\ paths that cross from one cube to another at vertices only. It is a consequence of the results from \cite{Z1} for d--simplicial complexes.

\begin{df}
	Let $M$ be a d-simplicial complex. A directed path $\alpha\in\vec{P}(M)$ is \emph{tame} if, for each $0\leq s<t\leq 1$, there exists a simplex $A\in S_{M}$ such that $\alpha|_{[s,t]}\subseteq |A|$, or a vertex $v\in V_M$ such that $v\in \alpha([s,t])$. Let $\vec{P}_t(M)_v^w\subseteq \vec{P}(M)_v^w$  be the subspace of tame paths from $v$ to $w$, $v,w\in K[0]$.
\end{df}

\begin{prp}[{\cite[Section 6]{Z1}}]\label{p:TamingDeformation}
	Let $M$ be a finite d-simplicial complex having no loops. There exists a d--map $R_M:|M|\to|M|$ such that
	\begin{enumerate}[(1)]
	\item{$R_M(v)=v$ for every vertex $v\in V_M$,}
	\item{for every $x\in |K|$, there exists a simplex $A$ such that $x,R_M(x)\in |A|$,}
	\item{for $v,w\in V_M$ and $\alpha\in \vec{P(M)_v^w}$ the path $R_M\circ \alpha$ is tame.}
	\end{enumerate}
\end{prp}

The tameness of paths on $\square$--sets is defined quite similarly:
\begin{df}
	Let $K$ be a $\square$-set. A path $\alpha\in\vec{P}(|K|)$ is \emph{tame} if, for each $0\leq s<t\leq 1$, there exists a cube $c\in K[n]$ such that $\alpha|_{[s,t]}\subseteq |c|$, or a vertex $v\in K[0]$ such that $v\in \alpha([s,t])$.
\end{df}

For the remaining part of this Section we assume that $K$ is a proper $\square$--set with height function.

\begin{prp}\label{p:TamePreserving}
	If $\alpha\in \vec{P}(\Tr(K))$ is tame (in the simplicial sense), then $F_K\circ \alpha\in \vec{P}(K)$ is also tame (in the cubical sense).
\end{prp}
\begin{proof}
	It follows immediately from \ref{p:TraingulationProperties}.(1) and (2).
\end{proof}

\begin{prp}\label{p:ConvexHomotopy}
	Let $X$ be a space and $f_1,f_2:X\to \vec{P}(K)_v^w$ continuous maps. Assume that, for every $t\in [0,1]$, there exists a cube $c$ such that $f_1(\alpha(t)),f_2(\alpha(t))\in |c|$. Then $f_1$ and $f_2$ are homotopic.
\end{prp}
\begin{proof}
	For $c\in K[n]$ the map  $|i_c|:\vec{I}^n\ni x\mapsto (c,x)\in |c|$ is a d-homeomorphism, since $K$ is proper. The homotopy between $f_1$ and $f_2$ is given by
	\[
		H_s(x)(t)=|i_c|((1-s)|i_c|^{-1}(f_1(\alpha(t)) + s|i_c|^{-1}(f_2(\alpha(t))).
	\]
	It is clear that this is well-defined (since convex combinations of d-paths in $\vec{I}^n$ are d-paths), continuous and does not depend on the choices of cubes $c$ for respective points.
\end{proof}

\begin{rem}
	If $K$ has height function, then convex combinations of natural paths are natural, so the analogue of \ref{p:ConvexHomotopy} holds also for maps into $\vec{N}(K)_v^w$.
\end{rem}

The main result of this Section is the following 
\begin{thm}\label{t:Taming}
	Assume that $K$ is a finite proper $\square$--set with height function. For $v,w\in K[0]$, the inclusion
	\[
		\vec{P}_{t}(K)_v^w \subseteq \vec{P}(K)_v^w
	\]
	is a homotopy equivalence.
\end{thm}
\begin{proof}
	The triangulation $\Tr(K)$ is finite and has no loops. Let $R_{\Tr(K)}:|\Tr(K)|\to|\Tr(K)|$ be a map satisfying the conditions of \ref{p:TamingDeformation}. For $\alpha\in\vec{P}(K)_v^w$ the path $F_{K}\circ R_{\Tr(K)}\circ (F_{\Tr(K)})^{-1} \circ \alpha$ is a d--path with endpoints $v,w$ (\ref{p:TamingDeformation}.(1)) and is tame (\ref{p:TamingDeformation}.(3), \ref{p:TamePreserving}). Therefore the map
	\[
		\vec{P}(K)_v^w\ni \alpha\mapsto F_K\circ R_{\Tr(K)}\circ (F_K)^{-1}\circ \alpha\in  \vec{P}_{t}(\Tr(K))_v^w
	\]
	is well-defined and is a homotopy inverse of the inclusion $\vec{P}_{t}(K)_v^w \subseteq \vec{P}(K)_v^w$. The latter statement follows from \ref{p:TamingDeformation}.(2), \ref{p:ConvexHomotopy} and \ref{p:TraingulationProperties}.(2).
\end{proof}

\section{Cube chain cover}

Let $K$ be a finite proper $\square$-set with height function and let $v,w\in K[0]$. Let
\[
	 \vec{N}_t(K)_v^w=(nat)^{-1}(\vec{P}_t(K)_v^w),
\]
(cf.\ \ref{e:Nat}) denote the space of natural tame paths; thanks to \ref{t:Taming} and \cite[2.16]{R1} the spaces $\vec{N}_t(K)_v^w$, $\vec{N}(K)_v^w$, $\vec{P}_t(K)_v^w$ and $\vec{P}(K)_v^w$ are all homotopy equivalent via normalization maps and inclusions. In this Section we construct a good closed cover of $\vec{N}_t(K)_v^w$ indexed with the poset of cube chains from $v$ to $w$.

Recall that a cube chain from $v$ to $w$ is a sequence of cubes $\bc=(c_1,\dots,c_l)$ such that $d^0(c_1)=v$, $d^1(c_l)=w$ and $d^1(c_i)=d^0(c_{i-1})$ for $1\leq i<l$.  Recall also that
\begin{equation}\label{e:DifferentialOfCubeChains}
	d_{k,A}(\bc)=(c_1,\dots,c_{k-1},d^0_{\bar{A}}(c_k), d^1_{A}(c_k),c_{k+1},\dots,c_l)\in\Ch(K)_v^w. 
\end{equation}
for $k\in \{1,\dots,l\}$, $A\cap \bar{A}=\emptyset$, $A\cup \bar{A}=\{1,\dots,\dim(c_k)\}$ and that $\leq$ is the transitive-reflexive closure of relations $d_{k,A}(\bc)\leq \bc$  on $\Ch(K)_v^w$.
For a cube chain $\bc=(c_1,\dots,c_l)$ we will write $l^{\bc}=l$, $n^\bc_i=\dim(c_i)$, $v^\bc_i=d^0(c_{i+1})=d^1(c_i)$ and $b^\bc_k=n^\bc_1+\dots+n^\bc_k$. The upper index $\bc$ will be omitted if it does not lead to confusion.

\begin{df}
	We say that a path $\alpha\in\vec{N}_t(K)_v^w$ \emph{lies in a cube chain $\bc\in\Ch(K)_v^w$}  if 
	\[
		\alpha([b^\bc_{i-1},b^\bc_i])\subseteq |c_i|
	\]
	for $i\in\{1,\dots,l\}$. Let $\vec{N}(K,\bc)\subseteq \vec{N}_t(K)$ be the subspace of tame paths lying in $\bc$.
\end{df}

Notice that $\alpha(b^\bc_i)=v^\bc_i$ for $\alpha\in \vec{N}(K,\bc)$, $i\in \{0,\dots,l^\bc\}$.

For $v,w,z\in K[0]$ and cube chains $\bc=(c_1,\dots,c_l)\in\Ch(K)_v^w$, $\bc'=(c'_1,\dots,c'_{l'})\in\Ch(K)_w^z$ we define \emph{the concatenation $\bc*\bc'\in \Ch(K)_v^z$} of $\bc$ and $\bc'$ by
\begin{equation}
	\mathbf{c}*\mathbf{c}'=(c_1,\dots,c_l,c'_1,\dots,c'_{l'}).
\end{equation}
The concatenation of paths induces a homeomorphism
\begin{equation}
	\vec{N}(K,\bc)\times \vec{N}(K,\bc')\subseteq \vec{N}(K,\bc*\bc').
\end{equation}

\begin{prp}\label{p:CoverProperties}
	Let $v,w\in K[0]$,  $\bc,\bc'\in\Ch(K)_v^w$.
	\begin{enumerate}[(1)]
	\item{$\vec{N}(K,\bc)$ is a contractible closed subspace of $\vec{N}_t(K)_v^w$.}
	\item{The intersection $\vec{N}(K,\bc)\cap \vec{N}(K,\bc')$ is either empty or there exists a cube chain $\bc\cap\bc'\in \Ch(K)_v^w$ such that $\vec{N}(K,\bc\cap\bc')=\vec{N}(K,\bc)\cap \vec{N}(K,\bc')$.}
	\item{$\vec{N}(K,\bc)\subseteq \vec{N}(K,\bc')$ if and only if $\bc\leq \bc'$.}
	\item{$\bigcup_{\bc\in\Ch(K)}\vec{N}(K,\bc)=\vec{N}_t(K)$.}
	\end{enumerate}
\end{prp}
\begin{proof}
	Assume that $h(v)=0$, $h(w)=n$, $\bc=(c_1,\dots,c_l)$, $\bc'=(c_1',\dots,c_{l'}')$. \\
	(1) Let $\gamma\in \vec{N}(K,\bc)$ be the diagonal of $\bc$, i.e.\  a path given by
	\[
		\gamma(t)=\left(c_i,\left(\tfrac{t-b^\bc_{i-1}}{n^\bc_i},\dots,\tfrac{t-b^\bc_{i-1}}{n^\bc_i}\right)\right)
	\]
	for $t\in[b^\bc_{i-1},b^\bc_i]$. The identity map of $\vec{N}(K,\bc)$ and the constant map with value $\gamma$ satisfy the assumptions of \ref{p:ConvexHomotopy}; therefore the space $\vec{N}(K,\bc)$ is contractible. Closedness of $\vec{N}(K,\bc)$ is clear.\\
	(2) Assume that $\alpha\in \vec{N}(K,\bc)\cap \vec{N}(K,\bc')$. The proof is by induction with respect to $n$. If $n=1$, then $\bc=(c_1)=(c_1')=\bc'$ since $c_1$ and $c'_1$ have the same extreme vertices. Assume that $n>1$ and $n^\bc_1< n^{\bc'}_1$. Since $\alpha(n^\bc_1)=v^\bc_1$ and $\alpha([0,n^{\bc'}_1])\subseteq |c'_1|$ then $v^\bc_1$ is a vertex of a cube $c_1'$, say $v^\bc_1=d_{A,\bar{A}}(c_1')$ for $A\cap \bar{A}=\emptyset$, $A\cup \bar{A}=\{1,\dots,n^{\bc'}_1\}$. We have $d^0(d^0_{A}(c'_1))=v=d^0(c_1)$ and $d^1(d^0_A(c_1'))=v^\bc_1=d^1(c_1)$; by the properness of $K$ this implies $c_1=d^0_A(c'_1)$. Furthermore, $\alpha({[n^\bc_1,n^{\bc'}_1]})\subseteq |d^1_{\bar{A}}(c_1')|$, and then $\alpha|_{[n_1^\bc,n]}$ lies in a cube chain $(d^1_{\bar{A}}(c_1'),c_2',\dots,c_{l'}')$. We have
	\begin{multline*}
		\vec{N}(K,\bc)\cap \vec{N}(K,\bc')=\vec{N}(K,\bc)\cap \vec{N}(K,(c_1,d^1_{\bar{A}}(c_1'),c'_2,\dots,c'_n))=\\
		\vec{N}(K,(c_1))\times (\vec{N}(K,(c_2,\dots,c_n))\cap \vec{N}(K,(d^1_{\bar{A}}(c_1'),c'_2,\dots,c'_n)))=\\
		\vec{N}(K,c_1*((c_2,\dots,c_n)\cap(d^1_{\bar{A}}(c_1'),c'_2,\dots,c'_n)))
	\end{multline*}
	since $(c_2,\dots,c_n)\cap(d^1_{\bar{A}}(c_1'),c'_2,\dots,c'_n)$ exists by the inductive assumption. The case $n_1^\bc=n_1^{\bc'}$ is similar, only the term $d^1_{\bar{A}}(c'_1)$ is dropped (since it is a vertex).\\
(3) Immediately from the definition, we have $\vec{N}(K,d_{k,A}(\bc))\subseteq \vec{N}(K,\bc)$. Assume that $\vec{N}(K,\bc)\subseteq \vec{N}(K,\bc')$; this is equivalent to the equation $\bc=\bc\cap\bc'$. Assume that the statement is true for all pairs of chains having length less than $n$. Using the argument from the previous point we obtain
\[
	\bc\cap\bc'=c_1*((c_2,\dots,c_n)\cap(d^1_{\bar{A}}(c_1'),c'_2,\dots,c'_n))
\]
By the inductive assumption, $(c_2,\dots,c_n)\leq(d^1_{\bar{A}}(c_1'),c'_2,\dots,c'_n)$ and therefore $\bc\leq\bc'$.\\
(4) Let $\alpha:[0,n]\to |K|$ be a natural tame path from $v$ to $w$. If $\alpha(t)$ is a vertex then $t$ is an integer. Thus, there is a finite sequence $0=k_0<\dots<k_l=n$ of integers such that $\alpha(k_i)=v_i$ is a vertex, and restrictions $\alpha|_{(k_{i-1},k_i)}$ do not contain vertices in its images. By the tameness of $\alpha$, for each $i\in\{1,\dots,l\}$ there exists a cube $c_i$ such that $\alpha([k_{i-1},k_i])\subseteq |c_i|$. Obviously $v_{i-1}$ and $v_i$ are vertices of $c_i$ and then there exist unique (by properness) subsets $A,B\subseteq \{1,\dots,\dim(c_i)\}$ such that $d_{A,\bar{A}}(c)=v_{i-1}$, $d_{B,\bar{B}}(c)=v_i$ (where $\bar{A},\bar{B}$ stand for suitable complements). The segment $\alpha|_{[k_{i-1},k_i]}$ lies in $c'_i:=d_{B,\vec{A}}(c_i)$ and $v_{i-1}=d^0(c'_i)$, $v_i=d^1(c'_i)$. Eventually, $\alpha$ lies in a cube chain $(c'_1,\dots,c'_l)$.
\end{proof}

This allows us to prove Theorem \ref{t:Cover} in the following special case. Let $\mathbf{fph}\square\Set_*^*$ be the category of finite proper $\square$--sets with height function and $\square$--maps that preserve height functions. For $(K,v,w)\in\mathbf{fph}\square\Set_*^*$ let $\varepsilon_{(K,v,w)}$ be the composition
\begin{equation}
	\varepsilon_{(K,v,w)}:\vP(K)_v^w \supseteq \vec{P}_t(K)_v^w \xrightarrow{nat} \vec{N}_t(K)_v^w \buildrel\simeq\over\leftarrow \hocolim_{\bc\in \Ch(K,\bc)} \vec{N}(K,\bc)\buildrel\simeq\over\to |\Ch(K)_v^w|.
\end{equation}
All the maps in the sequence and homotopy equivalences. This follows from \ref{t:Taming} for the left-most inclusion, \cite[2.16]{R1} for the $nat$ map, and from \cite[4.1]{S} and \ref{p:CoverProperties} for the remaining two maps. This implies that $\varepsilon_{(K,v,w)}$ is well-defined up to homotopy, and it is a homotopy equivalence. Furthermore, all the maps are functorial with respect to $(K,v,w)$. As a consequence, we obtain the following.

\begin{prp}\label{p:Cover}
	The maps $\varepsilon_{(K,v,w)}$ define a natural equivalence of functors
	\[
		\vP\simeq |\Ch|:\mathbf{fph}\square\Set_*^*\to\hTop.\qed
	\]
\end{prp}

\section{Permutahedra}

In this Section we study the posets of faces of products of permutohedra (see \cite[p.18]{Ziegler}), which play an important role in the description of the posets of cube chains on $\square$--sets. For a positive integer $n$, the face poset of the $(n-1)$--dimensional permutahedron is isomorphic to the poset of weak strict orderings on $\{1,\dots,n\}$, ordered by inclusion. A weak strict ordering is a reflexive and transitive relation such that any two elements are comparable, though, not necessarily anti-symmetric. Every weak strict ordering $\sqsubseteq$ on $\{1,\dots,n\}$ determines the unique surjective function
\begin{equation}
	f_\sqsubseteq: \{1,\dots,n\}\to \{1,\dots,k\},
\end{equation}
called \emph{the characteristic function} of $\sqsubseteq$, such that $i\sqsubseteq j$ if and only if $f_\sqsubseteq(i)\leq f_\sqsubseteq(j)$. In this Section, we use such functions rather than weak orderings.
Throughout the whole Section, $n$ stands for a fixed positive integer.

\begin{df}
	Let $O_n$ be a poset whose elements are surjective functions
	\[
		f:\{1,\dots,n\}\to \{1,\dots,k\},
	\]
	where $k=k(f)$ is a positive integer, and $f\preceq g$ if and only if there exists a non-decreasing function $h$ such that $g=h\circ f$.
\end{df}
Clearly, the function $f$ is determined uniquely. The poset $O_n$ has the greatest element: the constant function with the value $1$, which will be denoted by $f_n$. Let
\begin{equation}
	\partial O_n=O_n\setminus \{f_n\}.
\end{equation}

\begin{prp}\label{p:PermIso}
	The pairs $(|O^\Pi_n|,|\partial O^\Pi_{n}|)$ and $(D^{n-1},S^{n-2})$ are homeomorphic.
\end{prp}
\begin{proof}	
	As mentioned before, the poset $O^\Pi_n$ is isomorphic to the face lattice of $(n-1)$-dimensional permutohedron $P^{n-1}$ (cf.\ \cite[p.18]{Ziegler}), with $f_n$ corresponding to its body (i.e.\, the single $(n-1)$--dimensional cell) . Then the pair $(|O^\Pi_n|,|\partial O^\Pi_n|)$ is homeomorphic to $(P^{n-1},\partial P^{n-1})$ and hence to $(D^{n-1},S^{n-2})$.
\end{proof}
Let
\begin{itemize}
\item{$\Seq(n)$ be the set of all sequences of positive integers $\bn=(n_1,\dots,n_{l(\bn)})$ such that $n=n_1+\dots+n_{l(\bn)}$,}
\item{$b_k(\bn)=\sum_{j=1}^k n_j$, for $k\in\{0,\dots,l(\bn)\}$,}
\item{$f_{\bn}:\{1,\dots,n\}\to \{1,\dots,l(\bn)\}$ be the unique non-decreasing surjective function that takes value $k$ for exactly $n_k$ integers, i.e., $f_{\bn}(i)=k$ if and only if $b_{k-1}(\bn)<i\leq b_k(\bn)$,}
\item{$O_\bn:=(O_n)_{\preceq f_\bn}=\{f\in O_n:\; f\preceq f_\bn\}$.}
\end{itemize}
Notice that $O_{(n)}=O_n$. Whenever it would not lead to confusion, we will write $l$ or $b_k$ instead of $l(\bn)$ or $b_k(\bn)$.

Fix $\bn\in\Seq(n)$. For every $f\in O_{\bn}$, there exists a non-decreasing function $h:\{1,\dots,k(f)\}\to\{1,\dots,l\}$ such that $h\circ f=f_\bn$. Thus, for $k\in\{1,\dots,l\}$, $f$ determines a function
\begin{equation}\label{e:PreceqIso}
	f^k:\{1,\dots,n_k\}\xrightarrow{i\mapsto i+b_{k-1}} f_\bn^{-1}(k)\xrightarrow{f_\bn} h^{-1}(k)\xrightarrow{j\mapsto j+1-\min(h^{-1}(k))}\{1,\dots,r_k\},
\end{equation}
which is an element of $O_{n_k}$. On the other hand, a sequence
\[(f^k:\{1,\dots,n_k\}\to \{1,\dots,r_k\})\in O_{n_k},\]
for  $k\in \{1,\dots,l\}$, determines an element $f\in O_\bn$ such that
\begin{equation}
	f(i)=f^k(i-b_{k-1})+\sum_{j=1}^{k-1} r_j.
\end{equation}

These constructions give an isomorphism $O_{\mathbf{n}} \cong O_{n_1}\times\dots\times O_{n_l}$. Thus, $O_\bn$ is isomorphic to the face lattice of the product of permutahedra of dimensions $n_1-1,\dots,n_k-1$ respectively; as a consequence, there is a homeomorphism 
\begin{equation}\label{e:PermIso}
	(|O_{\mathbf{n}}|, |\partial O_{\mathbf{n}}|)\simeq (D^{n-l},S^{n-l-1}),
\end{equation}
where $\partial O_{\mathbf{n}}:=O_{\mathbf{n}}\setminus \{f_\bn\}$, which generalizes Proposition \ref{p:PermIso}.

The main goal of this Section is to construct a fundamental class of the pair (\ref{e:PermIso}) in simplicial homology \cite[p. 104]{H}; to achieve this, we need to introduce a notation for the simplices of $\cN O^\Pi_{n}$. Let
\begin{itemize}
	\item{$\Sigma_n$ be the set of permutations of $\{1,\dots,n\}$,}
	\item{$T^d_n$ be the set of functions $\tau:\{1,\dots,n-1\}\to \{0,\dots,d+1\}$ such that $\tau^{-1}(j)\neq\emptyset$ for $0<j\leq d$.}
\end{itemize}
For $\bn\in\Seq(n)$ define $\tau_\bn\in T^0_n$ by $\tau_\bn(i)=f_\bn(i+1)-f_\bn(i)$, and for $\tau\in T^0_n$ define $f^\tau\in O_n$ by
\begin{equation}
	f^\tau(i)=1+\sum_{i=1}^{k-1} \tau(i).
\end{equation}
Clearly, $f^{\tau_\bn}=f_\bn$, and this defines a bijection between $T^0_n$ and $\Seq(n)$. For $\tau\in T^d_n$ let 
\begin{equation}
	\Sigma_\tau=\{\sigma\in\Sigma_n:\; f^\tau\sigma = f^\tau\}.
\end{equation}
Let us enlist some properties of functions $f^\tau$:
\begin{prp} {\ } \label{p:PropertiesOfFST}
	\begin{enumerate}[(a)]
		\item{If $\tau,\psi\in T^0_n$, then $f^\tau\preceq f^\psi$ if and only if $\tau(i)\geq \psi(i)$ for all $i\in \{1,\dots,n-1\}$.}
		\item{$f^\tau\sigma\preceq f^\psi\varphi$ if and only if $f^\tau\preceq f^\psi$ and $\sigma\phi^{-1}\in\Sigma_\psi$.}
		\item{Every element of $O^\Pi_{n}$ has the form $f^\tau\sigma$ for $\sigma\in\Sigma_n$, $\tau\in T^0_n$.}
	\end{enumerate}
\end{prp}
\begin{proof}
	To prove (a), assume that $\tau(i)\geq \psi(i)$ for $\tau,\psi\in T^0_n$, $i\in\{1,\dots,n-1\}$. Clearly $\tau(i)=0$ implies that $\psi(i)=0$; thus, if $\tau(i)=\tau(j)$, then $\psi(i)=\psi(j)$. As a consequence, the formula $h(k)=(f^\psi\circ (f^\tau)^{-1})(k)$ defines a non-decreasing surjection such that $f^\psi=h\circ f^\tau$. On the other hand, if $\tau(i)=0$ and $\psi(i)=1$, then $f^\psi(i)\neq f^\psi(i+1)$ and $f^\tau(i)=f^{\tau}(i+1)$; therefore, $f^\psi$ cannot be written as a composition $h\circ f^\tau$.
	
If $f^\tau\preceq f^\psi$ and $\sigma\phi^{-1}\in\Sigma_\psi$, then
\[
	f^\psi\varphi=f^\psi \sigma\varphi^{-1}\varphi=f^\psi\sigma=hf^\tau\sigma,
\]
where $h$ is a function such that $f^\psi=hf^\tau$. If $f^\psi\varphi=hf^\tau\sigma$, then $f^\psi\varphi\sigma^{-1}=hf^\tau$ is a non-decreasing function. Hence $\varphi\sigma^{-1}\in \Sigma_\psi$ and $f^\psi=hf^\tau$, which implies (b). The point (c) is clear.
\end{proof}

Let us introduce the following maps between integers:
\begin{alignat*}{2}
	\sfd_k(i)&={\begin{cases}
		i & \text{for $i<k$}\\
		i+1 & \text{for $i\geq k$}
	\end{cases}} &\qquad  \sfs_k(i)&={\begin{cases}
		i & \text{for $i\leq k$}\\
		i-1 & \text{for $i> k$}
	\end{cases}}\\ 
	\sfp_k(i)&={\begin{cases}
		0 & \text{for $i\leq k$}\\
		1 & \text{for $i> k$}
	\end{cases}} & \qquad \sft_k(i)&={\begin{cases}
		k & \text{for $i= k+1$}\\
		k+1 & \text{for $i= k$}\\
		i & \text{for $i\neq k,k+1$.}
	\end{cases}}\\ 
\end{alignat*}
futhermore, for $r\leq s$ denote $\sft_r^s=\sft_r\sft_{r+1}\dots \sft_{s-1}$; clearly $\sft_r^s(i)=i$ for $i<r$ or $i>s$, $\sft_r^s(s)=r$, and $\sft_r^s(i)=i+1$ for $r\geq i<s$.

Every element $\tau\in T^d_n$ determines a sequence $(\sfp_0\circ \tau, \dots, \sfp_d\circ \tau)$ of $d+1$ elements of $T^n_0$. Thus, for $\sigma\in\Sigma_{n}$ and $\tau\in T^d_{n}$, we can define a $d$--simplex in $O^\Pi_{n}$:
\begin{equation}\label{e:SimplexA}
		a[\sigma,\tau]:=(f^{\sfp_0\tau}\sigma, f^{\sfp_1\tau}\sigma,\dots, f^{\sfp_d\tau}\sigma).
\end{equation}
Notice that  \ref{p:PropertiesOfFST}.(b) implies that $f^{\sfp_0\tau}\sigma\preceq f^{\sfp_1\tau}\sigma\preceq\dots\preceq f^{\sfp_d\tau}\sigma$.

\begin{prp}
	For $\sigma\in\Sigma_n$, $\tau\in T^d_n$ we have
\begin{enumerate}[(a)]
	\item{$\partial_i(a[\sigma,\tau])=a[\sigma,\sfs_i\tau]$, where $\partial_i$ denotes the simplex with $i$--th vertex skipped, $i\in\{0,\dots,d\}$,}
	\item{$a[\sigma,\tau]=a[\varphi,\psi]$ if and only if $\tau=\psi$ and $\sigma\varphi^{-1}\in \Sigma_\tau$.}
\end{enumerate}
\end{prp}
\begin{proof}
	This follows immediately from \ref{p:PropertiesOfFST} and the definition.
\end{proof}

For $\mathbf{n}\in\Seq(n)$ define
\begin{equation}
	\Sigma_{\bn}:=\{\sigma\in\Sigma_n:\; f_\bn\circ \sigma=f_\bn\},
\end{equation}
\begin{equation}
	T^d_\bn:=\{\tau\in T^d_n:\; \forall_{k=1,\dots,l(\bn)-1}\; \tau(b_k(\bn))=d+1\}.
\end{equation}
Notice that $T^0_\bn=\{\tau\in T^0_n:\; f^\tau\preceq f_\bn\}$, and $T^d_{\bn}=\{\tau\in T^d_n:\; \sfp_{d}\tau\in T^0_\bn\}$.

\begin{prp}
	Let $\mathbf{n}\in\Seq(n)$, $\sigma\in\Sigma_n$, $\tau\in T^d_n$. Then $a[\sigma,\tau]$ is a simplex in $\cN O^\Pi_{\mathbf{n}}$ if and only if $\sigma\in \Sigma_\bn$ and $\tau\in T^d_\bn$.
\end{prp}
\begin{proof}
	Clearly $a[\sigma,\tau]$ is a simplex of $\cN O_\bn$ if and only if $f^{\sfp_d\tau}\sigma\preceq f_\bn=f^{\tau_\bn}$. By \ref{p:PropertiesOfFST} this holds if and only if $\sigma\in\Sigma_{\bn}$ and $\sfp_d\tau \preceq \tau_\bn$.
\end{proof}

For $\sigma\in\Sigma_n$ let $\sgn(\sigma)\in\{\pm 1\}$ be the sign of permutation. For $\bn\in\Seq(n)$ and $\tau\in T^{n-l}_n$,  the image of $\tau$ contains $\{1,\dots,n-l\}$, and $\tau(b_j)=n-l$ for $j=1,\dots,l-1$. Then the restriction
\[
	\tau:\{1,\dots,n-1\}\setminus\{b_1,\dots,b_{l-1}\} \to \{1,\dots,n-l\}
\]	
is a bijection. Let $\varrho_{\mathbf{n}}: \{1,\dots,n-l\} \to \{1,\dots,n-1\}\setminus\{b_1,\dots,b_{l-1}\}$ be the (unique) increasing bijection. The composition $\tau\varrho_\bn$ is a permutation on $n-l$ letters, and we define \emph{the sign} of $\tau\in T^{n-l}_{\mathbf{n}}$ as
\begin{equation}\label{e:SignOfTau}
	\sgn_{\mathbf{n}}(\tau)=\sgn(\tau\varrho_\bn).
\end{equation}

Define a chain $g_\bn\in C_{n-l}(\cN O^\Pi_\bn)$ by
\begin{equation}
	g_{\mathbf{n}} = \sum_{\sigma\in\Sigma_{\mathbf{n}}}\sum_{\tau\in T^{n-l}_{\mathbf{n}}} \sgn(\sigma)\sgn_{\mathbf{n}}(\tau) a[\sigma,\tau].
\end{equation}
We will show that $g_\bn$ represents a fundamental class in $H_{n-l}(|O^\Pi_\bn|,|\partial O^\Pi_\bn|)$ (cf.\ \ref{e:PermIso}).

\begin{prp}\label{p:DeltaFormula}
	Let $\partial:C_{n-l}(\cN{O^\Pi_\bn})\to C_{n-l-1}(\cN{O^\Pi_\bn})$ be the differential of simplicial homological chain complex. Then
	\[
		\partial (g_\bn)=(-1)^{n-l}\sum_{\sigma\in\Sigma_{\bn}}\sum_{\tau\in T^{n-l(\bn)}_{\bn}} \sgn(\sigma)\sgn_{\bn}(\tau) a[\sigma,\sfs_{n-l}\tau].
	\]
	As a consequence, $g_\bn$ represents a generator of  $H^{n-l(\bn)}(\cN(O^\Pi_\bn), \cN(\partial O^\Pi_\bn))$.
\end{prp}
\begin{proof}
	We have
	\begin{multline*}
		\partial(g_\bn)=\sum_{\sigma\in \Sigma_{\bn}}\sum_{\tau\in T^{n-l}_\bn} \sgn(\sigma)\sgn_{\bn}(\tau)\sum_{j=0}^{n-l}(-1)^j \partial_j(a[\sigma,\tau])
		=\\
		 \sum_{\sigma,\tau}\sum_{j=0}^{n-l} (-1)^j  \sgn(\sigma)\sgn_{\bn}(\tau) a[\sigma,\sfs_j \tau]=
		\sum_{\sigma,\tau} \sgn(\sigma)\sgn_{\bn}(\tau) a[\sigma,\sfs_0\tau]+\\ 
		\sum_{j=1}^{n-l-1} (-1)^j \sum_{\sigma,\tau} \sgn(\sigma)\sgn_{\bn}(\tau) a[\sigma,\sfs_j\tau] +(-1)^{n-l}\sum_{\sigma,\tau} \sgn(\sigma)\sgn_{\bn}(\tau) a[\sigma,\sfs_{n-l}\tau].
	\end{multline*}
	For the first summand, we have $a[\sigma,\sfs_0\tau]=a[\sft_{\tau^{-1}(1)}\sigma,\sfs_0\tau]$, since $\sft_{\tau^{-1}(1)}\sigma\sigma^{-1}=\sft_{\tau^{-1}(1)}\in \Sigma_{\sfs_0\tau}$, therefore the first summand equals
	\begin{equation*}
		\sum_{\tau}\sgn_\bn(\tau)\sum_{\sigma:\; \sigma(\tau^{-1}(1))<\sigma(\tau^{-1}(1)+1)}\sgn(\sigma) a[\sigma,\sfs_0\tau]+\sgn(\sft_{\tau^{-1}(1)}\sigma)a[\sft_{(\tau^{-1}(1))}\sigma,\sfs_0\tau]=0.
	\end{equation*}
	For the second summand, for $j\in\{1,\dots,n-l-1\}$, we have $\sfs_j\tau=\sfs_j\sft_j\tau$, therefore
	\begin{multline*}
		\sum_{\sigma,\tau} \sgn(\sigma)\sgn_{\bn}(\tau) a[\sigma,\sfs_j\tau]=\\
		\sum_{\sigma}\sgn(\sigma)\sum_{\tau:\; \tau^{-1}(j)<\tau^{-1}(j+1)} \sgn_\bn(\tau) a[\sigma,\sfs_j\tau] + \sgn_\bn(\sft_j\tau) a[\sigma,\sfs_j \sft_j \tau]=0.
	\end{multline*}
	Eventually,  only the third summand remains. The maximal element of $O^\Pi_\bn$, namely $f_\bn$, does not appear as the vertex of $a[\sigma,\sfs_{n-l} \tau]$.  Thus, $\partial(g_\bn)\in C_{n-l-1}(\cN(\partial O^\Pi_\bn))$ and $g_\bn$ is a cycle, regarded as an element of $C_{n-l}(\cN(O^\Pi_\bn), \cN(\partial O^\Pi_\bn))$. Since the coefficients of $g_\bn$ on all simplices are $\pm 1$, it represents the generator in the homology group.
\end{proof}

Let $\mathbf{n}\in\Seq(n)$. For $k\in\{1,\dots,l\}$, $r\in\{1,\dots,n_k-1\}$ denote
\begin{equation}
	\bn[k,r]:=(n_1,\dots,n_{k-1},r,n_k-r,n_{k+1},\dots,n_l)\in\Seq(n).
\end{equation}

\begin{lem}\label{l:Sgn}
	For each $\tau\in T^{n-l}_{\mathbf{n}}$, there exists a unique pair of integers $k\in\{1,\dots,l\}$, $r\in\{1,\dots,n_k-1\}$ such that  $\tau(b_{k-1}+r)=n-l$, i.e., there is a bijection
\[
	T^{n-l}_\bn\ni \tau \mapsto \sfs_{n-l}\tau\in \coprod_{k=1}^{l} \coprod_{r=1}^{n_k-1} T^{n-l-1}_{\bn[k,r]}.
\]	
	Furthermore, if $\sfs_{n-l}\tau\in T^{n-l-1}_{\bn[k,r]}$, then
	\[
		\sgn_{\bn[k,r]}(\sfs_{n-l}\tau)=(-1)^{n-l-b_{k-1}-r+k-1}\sgn_{\bn}(\tau).
	\]
\end{lem}
\begin{proof}
	The existence of $k$ and $r$, as well as the fact that $\sfs_{n-l}\tau\in T^{n-l-1}_{\bn[k,r]}$, follows immediately from the definitions. It remains to prove that the signs of permutations $\sfs_{n-l}\tau\varrho_{\bn[k,r]}\in\Sigma_{n-l-1}$ and $\tau\varrho_\bn\in\Sigma_{n-l}$ (cf.\ \ref{e:SignOfTau}) differ by $(-1)^{n-l-b_{k-1}-r+k-1}$. Denote $v=b_{k-1}-(k-1)+r$; clearly
	\[
		\varrho_{\bn}(j)=\begin{cases}
			\varrho_{\bn[k,r]}(j) & \text{for $j\in\{1,\dots,v-1\}$},\\
			b_{k-1}+r & \text{for $j=v$,}\\
			\varrho_{\bn[k,r]}(j-1) & \text{for $j\in\{v+1,\dots,n-l\}$.}
		\end{cases}
	\]
	since $\im(\tau\varrho_{\bn[k,r]})\subseteq \{1,\dots,n-l-1\}$, we obtain
	\[
		\tau\varrho_{\bn}(j)=
		\begin{cases}
			\sfs_{n-l}\tau\varrho_{\bn[k,r]}(j) & \text{for $j\in\{1,\dots,v-1\}$}\\
			\tau(b_{k-1}+r)=n-l & \text{for $j=v$}\\
			\sfs_{n-l}\tau\varrho_{\bn[k,r]}(j-1) & \text{for $j\in\{v+1,\dots,n-l\}$}.\\
		\end{cases}
	\]
	Then
	\[
		\tau\varrho_{\bn}(j)=\begin{cases}
			\sfs_{n-l}\tau\varrho_{\bn[k,r]}\sft^{n-l}_v(j) & \text{for $j\in\{1,\dots,n-l-1\}$}\\
			n-l & \text{for $j=n-l$,}
		\end{cases}
	\]
	The permutations $\tau\varrho_\bn\in\Sigma_{n-l}$ and $\sfs_{n-l}\tau\varrho_{\bn[k,r]}\sft^{n-l}_v\in\Sigma_{n-l-1}$ have equal signs. As a consequence,
	\begin{multline*}
		\sgn_\bn(\tau)=\sgn(\tau\varrho_\bn)=\sgn(\sfs_{n-l}\tau\varrho_{\bn[k,r]}\sft^{n-l}_v) 
		=(-1)^{n-l-v}\sgn(\sfs_{n-l}\tau\varrho_{\bn[k,r]})=\\
		(-1)^{n-l-v}\sgn_{\bn[k,r]}(\sfs_{n-l}\tau)=(-1)^{n-l-b_{k-1}-r+k-1}\sgn_{\bn[k,r]}(\sfs_{n-l}\tau).\qedhere
	\end{multline*}
\end{proof}

Fix $\bn\in\Seq(n)$. For $k\in\{1,\dots,l\}$, $r\in\{1,\dots,n_k-1\}$, let
\begin{equation}
	S_\bn(k,r):=\{A\subseteq f_\bn^{-1}(k)=\{b_{k-1}+1,\dots,b_k\}:\; |A|=r \}.
\end{equation}

For $A\in S_\bn(k,r)$, let $\xi_{k,A}\in\Sigma_\bn$ be the permutation determined  by the following conditions:
\begin{itemize}
	\item{ $\xi_{k,A}(i)=i$ for $f_\bn(i)\neq k$,}
	\item{ $\xi_{k,A}|_A:A\to \{b_{k-1}+1, \dots, b_{k-1}+r\}$ be strictly increasing,}	
	\item{ $\xi_{k,A}|_{f^{-1}_\bn(k)\setminus A}:{f^{-1}_\bn(k)\setminus A}\to \{b_{k-1}+r+1, \dots, b_k\}$ be strictly increasing.}	
\end{itemize}

\begin{prp}\label{p:SignOfA}
	Assume that $\bn\in\Seq(n)$, $k\in\{1,\dots,l\}$, $r\in\{1,\dots,b_k-1\}$ and $A\in S_\bn(k,r)$. Then $\sgn(\xi_{k,A})=\sgn(A-b_{k-1})$, where $A-b_{k-1}:=\{i-b_k:\; i\in A \}$.
\end{prp}
\begin{proof}
	Assume that $A=(a_1<\dots<a_r)$. We have
	\[
		\xi_{k,A}= \sft^{a_r}_{b_{k-1}+r} \dots \sft^{a_2}_{b_{k-1}+2}\sft^{a_1}_{b_{k-1}+1},
	\]
	then $\sgn(\xi_{k,A})$ equals $-1$ to the power
	\[
		\sum_{j=1}^r a_j - \sum_{j=1}^r (b_{k-1}+j)=\sum_{j=1}^{r} (a_j-b_{k-1})-\sum_{j=1}^r j,
	\] 
	which coincides with the definition given in the introduction.
\end{proof}

\begin{lem}\label{l:Strata}
	For every $\bn\in\Seq(n)$, $k\in\{1,\dots,l\}$, $r\in\{1,\dots,n_k-1\}$ the function
	\[
		S_\bn(k,r)\times \Sigma_{\bn[k,r]}\ni (A,\omega) \mapsto \omega\xi_{k,A}\in \Sigma_{\bn}
	\]
	is a bijection.
\end{lem}
\begin{proof}
	For arbitrary $\sigma\in\Sigma_\bn$, let $A=\sigma^{-1}(\{b_{k-1}+1,\dots,b_{k-1}+r\})$. Then $\sigma\xi_{k,A}^{-1}\in\Sigma_{\bn[k,r]}$, and the inverse function is given by $\sigma\mapsto (A,\sigma\xi_{k,A}^{-1})$.
\end{proof}

Every permutation $\xi\in\Sigma_\bn$ induces an automorphism $\sigma^*:O_\bn\to O_\bn$ by the formula $\sigma^*(f)=f\sigma$; notice that $\xi_* a[\sigma,\tau]=a[\sigma\xi,\tau]$. Define an inclusion $\iota_{k,A}$ as a composition $O^\Pi_{\bn[k,r]}\subseteq O^\Pi_\bn \xrightarrow{(\xi_{k,A})^*}O^\Pi_\bn$.

\begin{prp}\label{p:GnDifferential}
	For $\bn\in\Seq(n)$ we have
	\[
		\partial g_\bn=\sum_{k=1}^l \sum_{r=1}^{n_k-1}\sum_{A\in S_\bn(k,r)} (-1)^{k+r+b_{k-1}+1} \sgn(A-b_{k-1}) (\iota_{k,A})_*(g_{\bn[k,r]}).
	\]
\end{prp}
\begin{proof}
	By \ref{p:DeltaFormula} we have
	\[
		\partial g_\bn=(-1)^{n-l}\sum_{\tau\in T^{n-l}_{\bn}}\sgn_{\bn}(\tau) \sum_{\sigma\in\Sigma_{\bn}} \sgn(\sigma) a[\sigma,\sfs_{n-l}\tau].
	\]
	Then, by Lemma \ref{l:Sgn}
	\[
		\partial g_\bn=\sum_{k=1}^l \sum_{r=1}^{n_k-1} \sum_{\psi\in T^{n-l-1}_{\bn[k,r]}} (-1)^{k+r+b_{k-1}+1} \sgn_{\bn[k,r]}(\psi) \sum_{\sigma\in\Sigma_{\bn}} \sgn(\sigma) a[\sigma,\psi].
	\]	
	Finally, using \ref{p:SignOfA} and \ref{l:Strata}, we obtain 
	\begin{multline*}
		\partial g_\bn=\sum_{k=1}^l \sum_{r=1}^{n_k-1} \sum_{\psi\in T^{n-l-1}_{\bn[k,r]}}  (-1)^{k+r+b_{k-1}+1} \sgn_{\bn[k,r]}(\psi) \sum_{\omega\in\Sigma_{\bn[k,r]}}\sum_{A\in S_\bn(k,r)} \sgn(\omega\xi_{k,A}) a[\omega\xi_{k,A},\psi]\\
		=\sum_{k=1}^l\sum_{r=1}^{n_k-1}\sum_{A\in S_\bn(k,r)}  (-1)^{k+r+b_{k-1}+1} \sgn(\xi_{k,A})\sum_{\psi\in T^{n-l-1}_{\bn[k,r]}}    \sum_{\omega\in\Sigma_{\bn[k,r]}} \sgn_{\bn[k,r]}(\psi)\sgn(\omega)  (\iota_{k,A})_*a[\omega,\psi]\\
		=\sum_{k=1}^l\sum_{r=1}^{n_k-1}\sum_{A\in S_\bn(k,r)} (-1)^{k+r+b_{k-1}+1} \sgn(A-b_{k-1}) (\iota_{k,A})_*(g_{\bn[k,r]}).\qedhere
	\end{multline*}
\end{proof}

\section{CW-decomposition} 
\label{s:CWDecomposition}

Throughout the whole Section, $K$ is a finite proper $\square$-set with height function $h$, and $v, w\in K[0]$ its vertices such that $h(v)=0$, $h(w)=n$. For $s\in \Z$, define a function
\[
	\sfu:\Z\ni i \mapsto \begin{cases}
		0 & \text{for $i>s$}\\
		1 & \text{for $i<s$}\\
		* & \text{for $i=s$.}
	\end{cases}
\]

\begin{prp}\label{p:IcMap}
	For every $c\in K[n]$, a map
	\[
		I_c:O_n\ni f \mapsto (d_{\sfu_1 f}(c),d_{\sfu_2 f}(c),\dots,d_{\sfu_{k(f)} f}(c))\in \Ch_{\leq (c)}(K),
	\]
	is an isomorphism of posets.
\end{prp}
\begin{proof}
	We will construct an inverse of $I_c$. For an arbitrary $\mathbf{b}=(b_1,\dots,b_l)\in \Ch_{\leq (c)}(K)$ and $k\in\{1,\dots,l\}$, there exists a unique, since $K$ is proper,  presentation $b_k=d_{A_k,B_k}(c)$ where $A_k$ and $B_k$ are disjoint subsets of $\{1,\dots,n\}$. For $X\subseteq \{1,\dots,n\}$ denote $\bar{X}=\{1,\dots,n\}\setminus X$. We have $B_1=\emptyset$ (since $d^0(b_1)=d^0(c)$), $A_l=\emptyset$ (since $d^1(b_l)=d^1(c)$), $B_{k+1}=\bar{A}_k$ (since $d^0(b_{k+1})=d^1(b_k)$), and $A_k\cup B_k\neq \{1,\dots,n\}$ (since $\dim(b_k)>0$). Therefore, we have a strictly increasing sequence
\[
	\emptyset=B_1\subsetneq \bar{A}_1=B_2\subsetneq \dots \subsetneq \bar{A}_{l-1}=B_l\subsetneq \bar{A}_{l}=\{1,\dots,n\},
\]
which defines a unique function $J_c(\mathbf{b}):\{1,\dots,n\}\to \{1,\dots,l\}$ such that $i\in \bar{A}_{f(i)}\setminus B_{f(i)}$. Clearly, $J_c$ is an inverse map of $I_c$.
\end{proof}

For $v,w\in K[0]$ and $\bc=(c_1,\dots,c_l)\in \Ch(K)_v^w$, define a map
\begin{equation}
	I_\bc:O_\bn\ni \;f\; \mapsto I_{c_1}(f^1)*\dots*I_{c_l}(f^l)\in \Ch_{\leq \bc}(K),
\end{equation}
where $f^i$ are defined as in (\ref{e:PreceqIso}).

\begin{prp}\label{p:SubposetIso}
	For every $\bc\in \Ch(K)_v^w$, $I_\bc$ is an isomorphism of posets.
\end{prp}
\begin{proof}
	Every $\mathbf{b}\in \Ch_{\leq\bc}(K)$ has a unique presentation as $\mathbf{b}^1*\dots*\mathbf{b}^l$, where $\mathbf{b}^k\in \Ch_{\leq (c_k)}(K)$. Define 
	\[
		J_\bc({\mathbf{b}})=(J_{c_1}(\mathbf{b}^1),\dots,J_{c_l}(\mathbf{b}^l))\in O_{n_1}\times\dots\times O_{n_l}\simeq O_{\bn},
	\]
	where $J_{c_k}$ is the map from the proof of \ref{p:IcMap}; $J_\bc$ is the inverse of $I_\bc$.
\end{proof}

\begin{prp}\label{p:CWp}
	Theorem \ref{t:CW} holds for finite proper bi-pointed $\square$--sets with height function.
\end{prp}
\begin{proof}
	By \ref{p:SubposetIso} and \ref{e:PermIso}, for every $\bc\in \Ch(K)$ the space $|\Ch_{<\bc}(K)_v^w|$ is homeomorphic to $S^{\dim(\bc)-1}$. Therefore, $\Ch(K)_v^w$, augmented with a minimal element $\emptyset$, is a CW poset in the sense of \cite{Bj}. Thus, by \cite[3.1]{Bj}, $|\Ch(K)_v^w|$ is a regular CW--complex with cells
\[
	A_\bc:= |\Ch_{\leq \bc}(K)_v^w|\setminus |\Ch_{< \bc}(K)_v^w|
\]
for $\bc\in \Ch(K)_v^w$. Every $\square$--map $f:K\to K'$ induces a morphism of posets $\Ch(K)_v^w\to \Ch(K')_{f(v)}^{f(w)}$, which maps $\Ch_{\leq\bc}(K)_v^w$ into $\Ch_{\leq f(\bc)}(K)_{f(v)}^{f(w)}$; therefore, $f_*:|\Ch(K)_v^w|\to|\Ch(K')_{f(v)}^{f(w)}|$ is cellular.
\end{proof}

Let $\Ch^{\leq n}(K)_v^w\subseteq \Ch(K)_v^w$ be the subset of cube chains having dimension less of equal $n$; clearly $|\Ch^{\leq n}(K)_v^w|$ is the $n$--skeleton of $|\Ch(K)_v^w|$. For a chain $\bc\in \Ch(K)_v^w$ of type $\bn$ and dimension $n$ let
\begin{equation}
	g_\bc := (I_\bc)_*(g_\bn)\in C_{n}(|\Ch^{\leq n}(K)_v^w|, |\Ch^{\leq n-1}(K)_v^w|).
\end{equation}
The group $H_n(|\Ch^{\leq n}(K)_v^w|, |\Ch^{\leq n-1}(K)_v^w|)$ is a free $\Z$-module generated by (simplicial, homological) chains $g_\bc$ for $\bc\in \Ch^{=n}(K)_v^w$.

\begin{prp}\label{p:CubeChainFunctoriality}
	Assume that $\bn\in\Seq(n)$, $k\in\{1,\dots,l=l(\bn)\}$, $r\in\{1,\dots,n_k-1\}$, and $A\in S_\bn(k,r)$. Then the diagram
\[
	\begin{diagram}
		\node{O_{\mathbf{n}[k,r]}}
			\arrow{e,t}{I_{d_{k,(A-b_{k-1})}(\mathbf{c})}}
			\arrow{s,l}{\iota_{k,A}}
		\node{\Ch(K)}
			\arrow{s,l}{=}
	\\
		\node{O_{\mathbf{n}}}
			\arrow{e,t}{I_{\mathbf{c}}}
		\node{\Ch(K)}
	\end{diagram}
\]
commutes.
\end{prp}
\begin{proof}
	It is sufficient to prove this in the case when $\bn=(n)$, $\bc=(c)$ and $k=1$. Let $f:\{1,\dots,n\}\to \{1,\dots,m\}$ be an element of $O_{\bn[k,r]}=O_{(r,n-r)}$. There exists a presentation $m=m_1+m_2$ such that $f(i)\leq m_1$ for $i\in\{1,\dots,r\}$ and $f(i)>m_1$ for $i\in\{r+1,\dots,n\}$. Furthermore,
	\[f^1:\{1,\dots,r\}\to \{1,\dots,m_1\}\] 
	is a restriction of $f$, 	and $f_2$ is the composition
	\[
		\{1,\dots,n-r\}\xrightarrow{i\mapsto i+r} \{r+1,\dots,n\}\xrightarrow{i\mapsto f(i)}\{m_1+1,\dots,m_2\}\xrightarrow{i\mapsto i-m_1}\{1,\dots,m_2\}.
	\]
	We will write $\xi$ instead of $\xi_{1,A}$. We have
	\[
		I_{(c)}(\iota_{1,A}(f))=I_{c}(f\xi)=(d_{\sfu_1f\xi}(c),\dots,d_{\sfu_mf\xi}(c))
	\]
	and
	\begin{multline*}
		I_{d_{1,A}((c))}(f)=I_{(d^0_{\bar{A}}(c),d^1_A(c))}(f)= I_{d^0_{\bar{A}}(c)}(f^1)*I_{d^1_{A}(c)}(f^2)=\\
		(d_{\sfu_1f^1} d^0_{\bar{A}}(c),\dots,d_{\sfu_{m_1}f^1} d^0_{\bar{A}}(c),
		d_{\sfu_{1} f^2} d^1_{A}(c),\dots,d_{\sfu_{m_2}f^2} d^1_{A}(c))
	\end{multline*}
	Since $\xi|_A$ is an increasing bijection $A\to\{1,\dots,r\}$, for $s\in\{1,\dots,r\}$ we have (cf.\ \ref{e:DfComposition}) $d_{\sfu_sf^1}d^0_{\bar{A}}(c)=d_g$, where $g|_{\bar{A}}=0=\sfu_sf\xi$, and $g|_A=f^1\xi=f\xi$. Hence, $g=\sfu_sf\xi$. The similar argument applies for $s\in\{r+1,\dots,n\}$.
\end{proof}

\begin{prp}\label{p:GcDiff}
	Let $\bc\in\Ch(K)_v^w$ be a cube chain of type $\bn$. Then
	\[
		\partial(g_\bc)=\sum_{k=1}^{l(\bn)} \sum_{r=1}^{n_k-1} \sum_{A\in S_\bn(k,r)} (-1)^{k+r+b_{k-1}+1}\sgn(A-b_{k-1}) g_{d_{k,A-b_{k-1}}(\bc)}
	\]
\end{prp}
\begin{proof}
	We will write $A'=A-b_{k-1}$. By \ref{p:CubeChainFunctoriality} and \ref{p:GnDifferential}, we have
	\begin{multline*}
		\partial(g_\bc)=\partial (I_\bc(g_\bn))=I_\bc(\partial g_\bn)
		\buildrel{(\ref{p:GnDifferential})}\over= I_\bc\left(\sum_{k=1}^{l}\sum_{r=1}^{n_k-1}\sum_{A\in S_\bn(k,r)}(-1)^{k+r+b_{k-1}+1}\sgn(A') \iota_{k,A}(g_{\bn[k,r]})\right)\\
		=\sum_{k=1}^{l}\sum_{r=1}^{n_k-1}\sum_{A\in S_\bn(k,r)}(-1)^{k+r+b_{k-1}+1}\sgn(A') I_{d_{k,A'}(\bc)}(g_{\bn[k,r]})\\
		=\sum_{k=1}^{l}\sum_{r=1}^{n_k-1}\sum_{A\in S_\bn(k,r)}(-1)^{k+r+b_{k-1}+1}\sgn(A') g_{d_{k,A'}(\bc)}.\qedhere
	\end{multline*}
\end{proof}

\begin{prp}\label{p:Diff}
	Theorem \ref{t:Diff} holds for finite proper bi-pointed $\square$--sets with height function.
\end{prp}
\begin{proof}
	By \ref{t:CW}, $H_n(|\Ch^{\leq m}(K)_v^w|, |\Ch^{\leq m-1}(K)_v^w|)$ is a free $\Z$-module generated by (simplicial, homological) chains $g_\bc$ for $\bc\in \Ch^{=m}(K)_v^w$, and the differentials are calculated in \ref{p:GcDiff}.
\end{proof}

\section{Non-looping length covering}

In this Section we generalize the results obtained for finite proper $\square$--sets with height function to a larger class of $\square$--sets that have proper non-looping length coverings. Recall that $\pcSqSet_*^*$ is the category of finite bi-pointed $\square$--sets having the proper non-looping covering.

For $K\in \pcSqSet_*^*$ and $n\geq 0$ define a $\square$--subset $\tilde{K}_n\subseteq \tilde{K}$ by $\tilde{K}_n[d]=K[d]\times \{0,\dots,n-d\}$ (cf.\ \ref{e:NonLoopingLengthCovering}). For every $n$, the formula $(K,v,w)\mapsto (\tilde{K}_n,(v,0),(w,n))$ defines a functor from $\pcSqSet_*^*$ to $\mathbf{fph}\square\Set_*^*$. 

Let $p:\tilde{K}\ni (c,k)\mapsto c\in K$ denote the obvious projection, and let $p_n=p|_{\tilde{K}_n}$. 

\begin{prp}\label{p:CoverPropertiesIso}
	Assume that $(K,v,w)\in \pcSqSet_*^*$. All maps in the sequence
	\[
		\coprod_{n\geq 0} \vec{N}(\tilde{K}_n)_{(v,0)}^{(w,n)}\subseteq  \coprod_{n\geq 0} \vec{P}(\tilde{K}_n)_{(v,0)}^{(w,n)}\subseteq \coprod_{n\geq 0} \vec{P}(\tilde{K})_{(v,0)}^{(w,n)}\xrightarrow{\coprod p_*} \vec{P}(K)_v^w
	\]
	are homotopy equivalences, and they are functorial with respect to $(K,v,w)$.
\end{prp}
\begin{proof}
	For the right-hand map it follows \cite[Proposition 5.3]{R3} and for the left-hand one from \cite[2.5]{R1}. Every directed path in $\vec{K}$ with endpoints in $\vec{K}_n$ lies in $\vec{K}_n$ so the middle inclusion is a homeomorphism.
\end{proof}

This criterion shows that, when proving Theorem \ref{t:Cover}, we can restrict to the case when $K$ if a finite proper $\square$-set with height function.

\begin{prp}\label{p:ChainCover}
	For an arbitrary $\square$-set $K$ and $v,w\in K[0]$ the sequence of $\square$--maps $\tilde{K}_n\subseteq \tilde{K}\xrightarrow{pr} K$ induces morphism of posets 
	\[
		\coprod_{n\geq 0} \Ch(\tilde{K}_n)_{(v,0)}^{(w,n)}\xrightarrow{\cong} \coprod_{n\geq 0} \Ch(\tilde{K})_{(v,0)}^{(w,n)} \xrightarrow{\simeq} \Ch(K)_v^w,
	\]
	which are isomorphisms.
\end{prp}
\begin{proof}
	The formula $\Ch(K)_v^w\ni (c_i)_{i=1}^l \mapsto ((c_i,\sum_{j=1}^{i-1} \dim(c_j))$ defines the inverse function and all the maps involved commute with $d_{K,A}$.
\end{proof}

\begin{proof}[\textbf{Proof of \ref{t:Cover}}]
	If $K$ is a $\square$--set with the proper non-looping covering, then, for $v,w\in K[0]$, there is a sequence of homotopy equivalences
	\[
		\vec{P}(K)_v^w\buildrel{\simeq}\over\longleftarrow \coprod_{n\geq 0} \vec{P}(\tilde{K}_n)_{(v,0)}^{(w,n)}\xrightarrow{\coprod \varepsilon_{\tilde{K}_n}}
		\coprod_{n\geq 0} |\Ch(\tilde{K}_n)_{(v,0)}^{(w,n)}|\xrightarrow{\simeq} |\Ch(K)_v^w|,
	\]
	which follow from \ref{p:CoverPropertiesIso}, \ref{p:Cover} and \ref{p:ChainCover} respectively.
\end{proof}

\begin{proof}[\textbf{Proof of \ref{t:CW} and \ref{t:Diff}}]
	By \ref{p:ChainCover}, these follow from \ref{p:CWp} and \ref{p:Diff}, respectively.
\end{proof}

\end{document}